\documentclass[12pt]{article}
\usepackage{amsmath}
\usepackage{amsfonts}
\usepackage{amssymb}
\usepackage{amsthm}
\usepackage{indentfirst}
\usepackage{fancyhdr}
\usepackage{graphicx}
\usepackage{float}

\usepackage{algorithm}
\usepackage[noend]{algpseudocode}

\title{Pattern Avoidance Over a Hypergraph}
\author{Maxwell Fishelson, Benjamin Gunby\\maxfish@mit.edu, bg570@connect.rutgers.edu}
\date{}

\newcommand{\p}[1]{\left({#1}\right)}
\newcommand{\card}[1]{\left|{#1}\right|}

\newcommand{\cC}{\mathcal{C}}
\newcommand{\cH}{\mathcal{H}}

\mathchardef\mhyphen="2D

\newtheorem{theorem}{Theorem}
\newtheorem{mydef}[theorem]{Definition}
\newtheorem{proposition}[theorem]{Proposition}

\newtheorem{lemma}[theorem]{Lemma}

\newtheorem*{thm}{Theorem}
\newtheorem*{lemma*}{Lemma}
\theoremstyle{remark}
\newtheorem*{remark}{Remark}
\newtheorem{step}{Step}

\numberwithin{theorem}{section}

\begin{document}

\maketitle

\begin{abstract}
A classic result of Marcus and Tardos (previously known as the Stanley-Wilf conjecture) bounds from above the number of $n$-permutations ($\sigma \in S_n$) that do not contain a specific sub-permutation. In particular, it states that for any fixed permutation $\pi$, the number of $n$-permutations that avoid $\pi$ is at most exponential in $n$. In this paper, we generalize this result.  We bound the number of avoidant $n$-permutations even if they only have to avoid $\pi$ at specific indices.  We consider a $k$-uniform hypergraph $\Lambda$ on $n$ vertices and count the $n$-permutations that avoid $\pi$ at the indices corresponding to the edges of $\Lambda$.  We analyze both the random and deterministic hypergraph cases.  This problem was originally proposed by Asaf Ferber.\\

When $\Lambda$ is a random hypergraph with edge density $\alpha$, we show that the expected number of $\Lambda$-avoiding $n$-permutations is bounded (both upper and lower) as $\exp(O(n))\alpha^{-\frac{n}{k-1}}$, using a supersaturation version of F\"{u}redi-Hajnal.\\

In the deterministic case we show that, for $\Lambda$ containing many size $L$ cliques, the number of $\Lambda$-avoiding $n$-permutations is $O\p{\frac{n\log^{2+\epsilon}n}{L}}^n$, giving a nontrivial bound with $L$ polynomial in $n$.  Our main tool in the analysis of this deterministic case is the new and revolutionary hypergraph containers method, developed in \cite{BMS} and \cite{ST}.
\end{abstract}

\section{Introduction}

Formally, the notion of pattern avoidance is defined as follows.

\begin{mydef}
An $n$-permutation $\sigma$ {\bf contains} a $k$-permutation $\pi$ iff there exist integers $1 \leq x_1 < x_2 < \cdots x_k \leq n$ such that
$$\pi(i)<\pi(j) \Leftrightarrow \sigma(x_i)<\sigma(x_j)$$
for all $i,j$.  Otherwise, we say $\sigma$ {\bf avoids} $\pi$.
\end{mydef}

In the late 1980s/early 1990s, Richard P. Stanley and Herbert Wilf independently conjectured that for every permutation $\pi$, there exists a constant $c_\pi$ such that the number of $n$-permutations avoiding $\pi < c_\pi^n$ for all $n$.  As there are $n!=e^{(1-o(1))n\log n)}$ permutations, this bound is non-trivial.  This conjecture was later proven by Marcus and Tardos \cite{MT} in 2003.

To generalize this result, we first generalize our notion of pattern avoidance in order to account for avoidance only at specific index sets.

\begin{mydef}
Let $\Lambda$ be a $k$-uniform hypergraph on vertex set $\{1,2,\cdots,n\}$.  We say an $n$-permutation $\sigma$ {\bf $\Lambda$-contains} a $k$-permutation $\pi$ iff there exist integers $1 \leq x_1 < x_2 < \cdots x_k \leq n$ such that
$$\pi(i)<\pi(j) \Leftrightarrow \sigma(x_i)<\sigma(x_j)$$
for all $i,j$ {\bf AND} $\{x_1,\cdots,x_k\} \in E(\Lambda)$.  Otherwise, we say $\sigma$ {\bf $\Lambda$-avoids} $\pi$.
\end{mydef}

In this paper, we analyze the generalized $\Lambda$-avoidance problem for both random hypergraphs and fixed hypergraphs, a problem originally posed by Asaf Ferber \cite{Personal}. When $\Lambda$ is a random hypergraph with edge density $\alpha$, we show that, for every permutation fixed $k\in\mathbb{Z}^+$ and $\pi\in S_k$, the number of $\Lambda$-avoiding $n$-permutations is $\exp(O(n))\alpha^{-\frac{n}{k-1}}$ in expectation.  We also show that, for fixed $\Lambda$, the number of $n$-permutations $\Lambda$-avoiding $\pi$ is $O\p{\frac{n\log^{2+\epsilon}n}{L}}^n$ for all $\epsilon > 0$, as long as $\Lambda$ is $k$-uniform and satisfies the following:\\

$\Lambda$ contains a collection of $L$-vertex cliques where each of the $n$ vertices belongs to at least $\delta(\Lambda) \geq 1$ cliques in the collection and at most $\Delta(\Lambda) = O(1)$.\\

We see that, for $L=n^{\Omega(1)}$, this bound is a non-negligible improvement on the $n^{(1-o(1))n}$ total $n$-permutations.\\

A few years after the proposal of Stanley-Wilf, in 1992, Zolt\'an F\"uredi and P\'eter Hajnal proposed a similar conjecture \cite{FH} that extended the notion of pattern-avoiding permutations to pattern-avoiding matrices.  Essentially, a $0\mhyphen 1$ matrix $A$ of size $n\times n$ contains a $0\mhyphen 1$ matrix $P$ of size $k \times k$ if there exists a $k \times k$ submatrix of $A$ that has 1-entries at all the locations where $P$ has 1-entries.  Formally,

\begin{mydef}
For a $0\mhyphen 1$ matrix $A$ of size $n\times n$ and a $k \times k$ $0\mhyphen 1$ matrix $P$, we say that $A$ {\bf contains} $P$ iff there exist row indices $1 \leq x_1 < x_2 < \cdots x_k \leq n$ and column indices $1 \leq y_1 < y_2 < \cdots y_k \leq n$ such that
$$P_{ij} = 1 \Rightarrow A_{x_iy_j} = 1$$
for all $i,j$.  Otherwise, we say $A$ {\bf avoids} $P$.  We note that, for $A$ to contain $P$, we don't require that $P$ be a submatrix of $A$, but that the 1-entries of $P$ be present in a submatrix of $A$.
\end{mydef}

The F\"uredi-Hajnal conjecture states that, if a $0\mhyphen 1$ matrix $A$ of size $n\times n$ avoids a permutation matrix $P_\pi$, it has $<c_Pn$ 1-entries for some constant $c_P$ in terms of $\pi$.  Progress was first made on these conjectures by Martin Klazar in 2000 \cite{K}, who showed that the F\"uredi-Hajnal conjecture implies the Stanley-Wilf conjecture.  Then, in 2004, Adam Marcus and G\'abor Tardos proved the F\"uredi-Hajnal conjecture \cite{MT}.  Combined with Klazar's arguments, a proof of the Stanley-Wilf conjecture was finally achieved.\\

This notion of pattern-avoiding matrices parallels that of pattern-avoiding permutations, as a permutation $\sigma$ contains a permutation $\pi$ if and only if the permutation matrix $P_\sigma$ contains the permutation matrix $P_\pi$.  The notion of $\Lambda$-avoidance can also be extended to this matrix context, where $A$ must only avoid $P$ on submatrices whose columns correspond to an edge in $\Lambda$.  Viewing pattern avoidance in this matrix context was the key to proving the Stanley-Wilf conjecture and will be one of the main insights in our analysis.\\

\section{Main Results}
When $\Lambda$ is a random hypergraph, we will prove the following bound.
\begin{theorem}\label{randomcasecor}
Let $k\in\mathbb{Z}$ with $k>1$, and take $\pi\in S_k$. Then there is some constant $C=C(\pi)$ such that if $\Lambda$ is the $k$-uniform Erd\H{o}s-R\'{e}nyi random hypergraph on $n$ vertices with edge probability $\alpha\in (0,1]$, then the expected number of $\sigma\in S_n$ that $\Lambda$-avoid $\pi$ is at most
\[\exp(Cn)\alpha^{-\frac{n}{k-1}}.\]
Furthermore, for $\alpha\geq n/\binom{n}{k}$, this bound is sharp to within an exponential factor; that is, up to a modification in $C$ (making it potentially negative).
\end{theorem}
\begin{remark}
When $\alpha=n/\binom{n}{k}$, note that this theorem gives a lower bound of $\exp(Cn)\cdot n!$. Thus $n!$ is correct to an exponential factor when $\alpha<n/\binom{n}{k}$ (the answer can only increase when $\alpha$ decreases). Thus we have successfully managed bound the expectation within an exponential for all values of $\alpha$.
\end{remark}
As we will see in Section \ref{cordeduction}, due to linearity of expectation, Theorem \ref{randomcasecor} reduces to bounding the number of permutations containing few copies of $\pi$, for which we will require bounds on the maximal number of ones in a $0\mhyphen 1$ matrices containing few copies of the permutation matrix $A_{\pi}$. Both of these bounds may be of independent interest as they give sharp approximations up to respectively an exponential and a constant.

\begin{theorem}\label{01matrices}
Let $k\in\mathbb{Z}^+$, $\pi\in S_k$, and let $A_{\pi}$ be the $k\times k$ permutation matrix corresponding to $\pi$. There exist constants $C=C(\pi)$ and $C'=C'(\pi)>0$ such that if $M$ is a $0\mhyphen 1$ matrix of size $n\times n$ containing $a$ ones, with $Cn\leq a\leq n^2$, then $M$ contains at least $C'\frac{a^{2k-1}}{n^{2k-2}}$ copies of $A_{\pi}$. Furthermore, for $a$ in the given range, this bound is sharp to within a constant factor (depending on $\pi$), in the sense that for any $a$ one can always find an $M$ that attains this lower bound to within a constant factor.
\end{theorem}
\begin{remark}
This theorem can be thought of as a `supersaturation' version of F\"uredi-Hajnal. Indeed, while F\"uredi-Hajnal states that with $>Cn$ ones at least one copy of $A_{\pi}$ is forced, Theorem \ref{01matrices} gives a bound on the number of copies of $A_{\pi}$ that are forced by any number of ones. We will see in Section \ref{01boundsection} how to prove Theorem \ref{01matrices} by bootstrapping F\"uredi-Hajnal.
\end{remark}
In a similar but more complicated way to the deduction of Stanley-Wilf from F\"uredi-Hajnal, we will be able to show the following.
\begin{theorem}\label{permsupersaturation}
Let $k\in\mathbb{Z}^+$, $k>1$ and $\pi\in S_k$. There exists some constant $C=C(\pi)$ and $c=c(\pi)$, $C>c\in\mathbb{R}$ such that for all $m,n\in\mathbb{Z}^{\geq 0}$ with $m\leq\binom{n}{k}$, letting $S_n(m,\pi)$ be the number of permutations in $S_n$ containing at most $m$ copies of $\pi$, we have that
\[\exp(cn)\cdot\max\left(1,\left(\frac{m}{n}\right)^{\frac{n}{k-1}}\right)\leq S_n(m,\pi)\leq \exp(Cn)\cdot\max\left(1,\left(\frac{m}{n}\right)^{\frac{n}{k-1}}\right).\]
\end{theorem}
\begin{remark}
Notice that $c$ is potentially negative, and thus the lower bound in Theorem \ref{permsupersaturation} is only nontrivial if $m>n$. Further note that the theorem simply reduces to Stanley-Wilf when $m=0$.
\end{remark}

In Section \ref{cordeduction}, we will make the easy deduction of Theorem \ref{randomcasecor} as a corollary of Theorem \ref{permsupersaturation}. In Section \ref{01boundsection}, we will prove Theorem \ref{01matrices} by bootstrapping F\"uredi-Hajnal. Finally, in Section \ref{randomproof} we will prove Theorem \ref{permsupersaturation}.\\

\begin{remark}
It is natural to ask whether in addition to the expectation result in Theorem \ref{randomcasecor}, one can also derive a concentration result on the number of $\sigma$ that $\Lambda$-avoid $\pi$. However, this at least does not seem to us to be obvious. For example, if one tries computing the variance of the random variable (in order to for example apply Chebyshev's inequality), one immediately runs into difficulties, which we will now describe.

Notice that the expected number of $\sigma$ that $\Lambda$-avoid $\pi$ is a sum of indicator random variables, one for each $\sigma\in S_n$, which are $1$ if and only if that $\sigma$ $\Lambda$-avoids $\pi$. In a variance computation, we would need to compute the covariance of these indicators for $\sigma$ and $\sigma'$. This computation involves simultaneously keeping track of the number of copies of $\pi$ in $\sigma$ and $\sigma'$, along with the number of index sets (of size $k$) at which $\sigma$ and $\sigma'$ simultaneously contain a copy of $\pi$. The interaction of these three quantities seems difficult to deal with. 

Of course, Markov's inequality combined with Theorem \ref{randomcasecor} shows that an upper bound of the type given in Theorem \ref{randomcasecor} holds with probability $1-e^{-n}$. However, since this bound potentially involves changing the constant $C$, it is quite weak, and given the difficulties above it seems like proving stronger bounds will likely require new ideas.
\end{remark}
\vspace{12pt}

We will also consider the case when $\Lambda$ is a fixed graph with particular structure. In particular, we will show the following.
\begin{theorem}\label{fixedtheorem}
For every permutation $\pi$ and any $\epsilon > 0$, the number of $n$-permutations $\Lambda$-avoiding $\pi$ is $O\p{\frac{n\log^{2+\epsilon}n}{L}}^n$ as long as $\Lambda$ is $k$-uniform and satisfies the following:\\

$\Lambda$ contains a collection of $L$-vertex cliques where each of the $n$ vertices belongs to at least $\delta(\Lambda) \geq 1$ cliques in the collection and at most $\Delta(\Lambda) = O(1)$.
\end{theorem}

In Sections \ref{fixedcase} to \ref{SRB}, we will prove Theorem \ref{fixedtheorem}.  The main tool in our analysis will be the hypergraph containers method.  The containers method enables us to distribute the vertices of a hypergraph into containers such that every independent set in the hypergraph belongs to one of the containers.  We can apply this method recursively, breaking each container down further into more containers in a branching fashion, to bound the total number of independent sets in a hypergraph.\\

We will set up a hypergraph whose vertices represent the 1-entries in a matrix and whose edges represent the entries in a submatrix containing $P_\pi$ with columns $\in E(\Lambda)$.  In this context, independent sets correspond to $\Lambda$-avoiding matrices.  Using the hypergraph containers method, we bound the number of permutation-matrix independent sets, utilizing F\"uredi-Hajnal to show that the conditions needed to apply the method hold.\\

In Section \ref{fixedcase} we introduce this fixed $\Lambda$ case.  We motivate the constraint that $\Lambda$ contains cliques of size $\text{poly}(n)$ by demonstrating that $\Lambda$ with maximal clique of constant size can contain $\Theta(n^k)$ edges and still be avoided by almost all $n$-permutations.  In Section \ref{formulation}, we establish the matrix/hypergraph formulation of the problem.  In Section \ref{HC}, we formally introduce the hypergraph containers lemma and investigate the necessary conditions to apply the lemma in a recursive branching fashion.  In Sections \ref{RL} and \ref{SRB}, we verify that these conditions are met using two additional lemmas, completing the proof of Theorem \ref{fixedtheorem}.\\

Many of the arguments in these sections parallel those presented in a paper \cite{FMS} by Asaf Ferber, Gweneth Anne McKinley, and Wojciech Samotij.  Additionally, the application of the hypergraph container lemma in a recursive branching fashion is adopted from a paper \cite{MS} by Morris and Saxton.\\

Lastly, in Section \ref{conclusion}, we will compare Theorems \ref{randomcasecor} and \ref{fixedtheorem} and summarize our results.

\section{Linearity of Expectation}\label{cordeduction}

Fix an integer $k>1$ and suppose $\Lambda$ is a random $k$-uniform hypergraph on $[n]$ with each edge chosen independently at random with edge probability $\alpha\in (0,1]$. In this case, we may simplify the problem by making use of linearity of expectation. In particular, let us define
\[Av_{n,\Lambda}(\pi):=\{\sigma\in S_n:\sigma\text{ }\Lambda\text{-avoids }\pi\}.\]
Then by linearity of expectation, we have that
\[\mathbb{E}_{\Lambda}[|Av_{n,\Lambda}(\pi)|]=\displaystyle\sum_{\sigma\in S_n}\Pr[\sigma\text{ }\Lambda\text{-avoids }\pi].\]
This latter probability is simply the probability that none of the copies of $\pi$ in $\sigma$ correspond to edges of $\Lambda$, which is $(1-\alpha)^{\#\text{ of copies of }\pi\text{ in }\sigma}$. Therefore,
\begin{equation}\label{linexp}
\mathbb{E}_{\Lambda}[|Av_{n,\Lambda}(\pi)|]=\displaystyle\sum_{\sigma\in S_n}(1-\alpha)^{\#\text{ of copies of }\pi\text{ in }\sigma}.
\end{equation}
Thus bounds on the number of permutations containing few copies of $\pi$, as given in Theorem \ref{permsupersaturation}, will give us bounds on our desired quantity $\mathbb{E}_{\Lambda}[|Av_{n,\Lambda}(\pi)|]$. We now make this argument rigorous.

\begin{proof}[Deduction of Theorem \ref{randomcasecor} from Theorem \ref{permsupersaturation}]
We first prove the upper bound. By (\ref{linexp}),
\begin{align*}
\mathbb{E}_{\Lambda}[|Av_{n,\Lambda}(\pi)|] & =\displaystyle\sum_{\sigma\in S_n}(1-\alpha)^{\#\text{ of copies of }\pi\text{ in }\sigma} \\ & \leq\displaystyle\sum_{m=0}^{\binom{n}{k}}(1-\alpha)^m\cdot |\{\sigma\in S_n:\sigma\text{ contains at most }m\text{ copies of }\pi\}|.
\end{align*}
By Theorem \ref{permsupersaturation}, there exists $C=C(\pi)$ such that this is at most
\begin{align*}
\displaystyle\sum_{m=0}^{\binom{n}{k}}(1-\alpha)^m\exp(Cn)\cdot & \max\left(1,\left(\frac{m}{n}\right)^{\frac{n}{k-1}}\right) =\displaystyle\sum_{m=0}^n(1-\alpha)^m\exp(Cn) \\ & +\displaystyle\sum_{m=n+1}^{\binom{n}{k}}(1-\alpha)^m\exp(Cn)\left(\frac{m}{n}\right)^{\frac{n}{k-1}} \\ & \leq (n+1)\exp(Cn)+\frac{\exp(Cn)}{n^{\frac{n}{k-1}}}\displaystyle\sum_{m=n+1}^{\binom{n}{k}} (1-\alpha)^m m^{\frac{n}{k-1}} \\ & \leq (n+1)\exp(Cn)+\binom{n}{k}\frac{\exp(Cn)}{n^{\frac{n}{k-1}}}\cdot\displaystyle\max_{m\in\mathbb{R}^+}(1-\alpha)^m m^{\frac{n}{k-1}} \\ & \leq \exp((C+k)n)\left(1+n^{-\frac{n}{k-1}}\cdot\displaystyle\max_{m\in\mathbb{R}^+}(1-\alpha)^m m^{\frac{n}{k-1}}\right) \\ & \leq\exp((C+k)n)\left(1+n^{-\frac{n}{k-1}}\cdot\displaystyle\max_{m\in\mathbb{R}^+}e^{-\alpha m} m^{\frac{n}{k-1}}\right),
\end{align*}
where we are simply bounding our sum by its number of terms times its maximum term, and using the trivial bounds $\binom{n}{k}\leq n^k<(e^n)^k=e^{kn}$, $n+1\leq kn<e^{kn}$ for $n\geq 1$ and $k\geq 2$, and $1-\alpha\leq e^{-\alpha}$. Now, by taking the logarithm and differentiating with respect to $m$, we see that $e^{-\alpha m} m^{\frac{n}{k-1}}$ is maximized when $-\alpha+\frac{n}{(k-1)m}=0$, or rearranging, $m=\frac{n}{(k-1)\alpha}$. Substituting, we have that
\[\displaystyle\max_{m\in\mathbb{R}^+}e^{-\alpha m} m^{\frac{n}{k-1}}=\left(\frac{n}{e(k-1)\alpha}\right)^{\frac{n}{k-1}}.\]
Putting this into the calculation from earlier,
\begin{align*}
\mathbb{E}_{\Lambda}[|Av_{n,\Lambda}(\pi)|] & \leq\exp((C+k)n)\left(1+n^{-\frac{n}{k-1}}\left(\frac{n}{e(k-1)\alpha}\right)^{\frac{n}{k-1}}\right) \\ & =\exp((C+k)n)\left(1+\left(\frac{1}{e(k-1)\alpha}\right)^{\frac{n}{k-1}}\right) \\ & \leq\exp((C+k)n)\left(1+\alpha^{-\frac{n}{k-1}}\right) \\ & \leq\exp((C+k+1)n)\alpha^{-\frac{n}{k-1}},
\end{align*}
the last step being because $\alpha\leq 1$. Replacing $C+k+1$ by $C$, we have deduced the upper bound.

For the lower bound, suppose now $\alpha\geq n/\binom{n}{k}$. Let $m=\left\lceil\frac{n}{\alpha}\right\rceil$ in the lower bound of Theorem \ref{permsupersaturation}. Since $\alpha\geq n/\binom{n}{k}$, we have $m\leq\binom{n}{k}$, so this is valid. We obtain that there are at least $\exp(C'n)\alpha^{-\frac{n}{k-1}}$ permutations in $S_n$ containing at most $\left\lceil\frac{n}{\alpha}\right\rceil$ copies of $\pi$ for some $C'=C'(\pi)$. Thus, by (\ref{linexp}),
\begin{align*}
\mathbb{E}_{\Lambda}[|Av_{n,\Lambda}(\pi)|] & \geq\exp(C'n)\alpha^{-\frac{n}{k-1}}(1-\alpha)^{\left\lceil\frac{n}{\alpha}\right\rceil} \\ & \geq\exp(C'n)\alpha^{-\frac{n}{k-1}}\exp\left(-\frac{\alpha}{1-\alpha}\cdot\frac{2n}{\alpha}\right) \\ & \geq\exp\left(\left(C'-\frac{2}{1-\alpha}\right)n\right)\alpha^{-\frac{n}{k-1}},
\end{align*}
where in the second line we used the inequality $\log(1-\alpha)\geq-\frac{\alpha}{1-\alpha}$ (an easy consequence of Taylor expansion) and $\left\lceil\frac{n}{\alpha}\right\rceil\leq\frac{2n}{\alpha}$ (immediate as $\alpha\leq 1$ and $n\geq 1$). This proves the lower bound with a constant of $C=C'-4$ when $\alpha\leq\frac{1}{2}$.

With $\alpha\geq\frac{1}{2}$, note that $\mathbb{E}_{\Lambda}[|Av_{n,\Lambda}(\pi)|]\geq 1\geq\exp(-n)\alpha^{-\frac{n}{k-1}}$ (as either the all-increasing or all-decreasing permutation avoids $\pi$ over any hypergraph), so the constant $-1$ suffices. So letting $C=\min(C'-4,-1)$ is sufficient to prove the lower bound, completing our argument.
\end{proof}

\section{Bounds on $0\mhyphen 1$ Matrices}\label{01boundsection}
As in the proof strategy of \cite{MT}, before we prove our result for permutations we first pass to the domain of $0\mhyphen 1$ matrices. Since we would like to bound the number of permutations with few copies of $\pi$, we first show that a matrix $M$ that contains few copies of the corresponding permutation matrix $A_{\pi}$ must have few ones.

The technique we use to prove Theorem \ref{01matrices} is a classic method for proving supersaturation results; that is, we show that a random submatrix of $M$ will with non-negligible probability contain at least one copy of $A_{\pi}$, so all of $M$ must contain several copies of $A_{\pi}$.

In \cite{MT}, Marcus and Tardos famously proved the following result, previously known as the F\"uredi-Hajnal conjecture.

\begin{theorem}[Marcus-Tardos]\label{FurHaj}
There exists a constant $c_{\pi}$ such that for all $n$, any $0\mhyphen 1$ matrix of size $n\times n$ containing at least $c_{\pi}n$ ones contains a copy of $A_{\pi}$.
\end{theorem}

From this, we can immediately deduce the following (extremely weak) supersaturation result, which we will bootstrap using sampling into our stronger results.

\begin{lemma}\label{easybound}
With $c_{\pi}$ as in Theorem \ref{FurHaj}, any $0\mhyphen 1$ matrix of size $n\times n$ with $m$ ones contains at least $m-c_{\pi}n$ copies of $A_{\pi}$.
\end{lemma}
\begin{proof}[Proof of Lemma \ref{easybound}]
We proceed by induction on $m$. For $m\geq c_{\pi}n$ the result is trivial.

Now take $m>c_{\pi} n$ and assume Lemma \ref{easybound} for $m-1$. Take a $0\mhyphen 1$ matrix $M$ of size $n\times n$ with $m$ ones. Now, by Theorem \ref{FurHaj}, it contains a copy of $A_{\pi}$. Let $M'$ be the matrix given by $M$ with one of the ones in this copy of $A_{\pi}$ changed to a $0$. By the inductive hypothesis, $M'$ contains at least $m-c_{\pi}n-1$ copies of $A_{\pi}$. But in going from $M$ to $M'$ we eliminated at least one copy of $A_{\pi}$ by design, so $M$ must have at least $m-c_{\pi}n$ copies of $A_{\pi}$, finishing the induction.
\end{proof}

We are now ready to prove Theorem \ref{01matrices}.

\begin{proof}[Proof of Theorem \ref{01matrices}]
Take $k\in\mathbb{Z}^+$, $\pi\in S_k$, and let $c_{\pi}$ be as given by Theorem \ref{FurHaj} and Lemma \ref{easybound}. Let $M$ be a $0\mhyphen 1$ matrix of size $n\times n$ containing $a$ ones, $Cn\leq a\leq n^2$, with $C$ to be chosen later.

Take an $r$ by $r$ submatrix $R$ of $M$, with $r$ to be chosen later. Let the density of ones in $R$ (that is, the number of ones in $R$ divided by $r^2$) be $1(R)$. Similarly, let the density of $A_{\pi}$ in $R$ (that is, the number of copies of $A_{\pi}$ in $R$ divided by $\binom{r}{k}^2$) be $\pi(R)$. Define $1(M)$ and $\pi(M)$ similarly; in particular, $1(M)=\frac{a}{n^2}\geq\frac{C}{n}$ by assumption.

In this notation, Lemma \ref{easybound} tells us that 
\[\binom{r}{k}^2\pi(R)\geq r^21(R)-c_{\pi}r,\]
or rearranging,
\begin{equation}\label{Rbound}
1(R)\leq\frac{\binom{r}{k}^2}{r^2}\pi(R)+\frac{c_{\pi}}{r}.
\end{equation}

Now, let $R$ be a \emph{random} $r\times r$ submatrix of $M$ (we choose a random subset of size $r$ of the rows and similarly for the columns). Now, for each copy of $A_{\pi}$ in $M$ (defined by $k$ rows and $k$ columns), there is a $\frac{\binom{r}{k}^2}{\binom{n}{k}^2}$ probability that all rows and columns corresponding to this copy of $A_{\pi}$ are chosen to be in $R$. Thus the expected number of copies of $A_{\pi}$ in $R$ is $\frac{\binom{r}{k}^2}{\binom{n}{k}^2}$ times the number of copies of $A_{\pi}$ in $M$, and therefore
\[\mathbb{E}[\pi(R)]=\pi(M).\]

Similarly, each entry of $M$ has equal probability of appearing in $R$, and so
\[\mathbb{E}[1(R)]=1(M)\geq\frac{a}{n^2},\]
as by assumption $M$ has at least $a$ ones.

Now that we have $1(M)$ and $\pi(M)$ expressed in terms of $1(R)$ and $\pi(R)$, (\ref{Rbound}) applied to $R$ will give an inequality between $1(M)$ and $\pi(M)$. Explicitly,
\begin{equation}\label{1pirelation}
1(M)=\mathbb{E}[1(R)]\leq\mathbb{E}\left[\frac{\binom{r}{k}^2}{r^2}\pi(R)+\frac{c_{\pi}}{r}\right]=\frac{\binom{r}{k}^2}{r^2}\pi(M)+\frac{c_{\pi}}{r}.
\end{equation}
We now have lower bounded $\pi(M)$ (which is a scaling of the number of copies of $A_{\pi}$ in $M$) in terms of the number of ones in $M$. It only remains to optimize the value of our parameter $r$ in order to obtain the desired lower bound.

Clearly, we need $1(M)>\frac{c_{\pi}}{r}$ for (\ref{1pirelation}) to give any bound on $\pi(M)$ at all, so we choose $r=\left\lfloor\frac{3c_{\pi}}{1(M)}\right\rfloor$. We require $r\leq n$ (so that we can sample $r\times r$ submatrices), but this holds as long as $1(M)\geq\frac{3c_{\pi}}{n}$; that is, $M$ has at least $3c_{\pi}n$ ones. Thus taking $C=3c_{\pi}$ in the statement of Theorem \ref{01matrices} is sufficient to satisfy $r\leq n$. (This will be our only restriction on $C$.)

Note that since $c_{\pi}\geq 1$ and $1(M)\leq 1$, we have $r\geq\frac{2c_{\pi}}{1(M)}$. Thus $1(M)-\frac{c_{\pi}}{r}\geq\frac{1(M)}{2}$. Substituting into (\ref{1pirelation}),
\begin{align*}
\frac{1(M)}{2} & \leq 1(M)-\frac{c_{\pi}}{r} \\ & \leq\frac{\binom{r}{k}^2}{r^2}\pi(M).
\end{align*}
Now, since $n\geq r$, and the function $\frac{\binom{a}{k}}{a^k}$ is increasing for $a>k$ (and $0$ on integers less than $k$), we have that $\binom{r}{k}\leq\frac{r^k}{n^k}\binom{n}{k}$. Substituting this yields
\begin{align*}
1(M) & \leq 2\frac{\binom{r}{k}^2}{r^2}\pi(M) \\ & \leq 2\frac{r^{2k-2}}{n^{2k}}\binom{n}{k}^2\pi(M) \\ & \leq 2\frac{(3c_{\pi})^{2k-2}}{(1(M))^{2k-2}n^{2k}}\binom{n}{k}^2\pi(M).
\end{align*}
Letting $C'=C'(\pi):=\frac{1}{2(3c_{\pi})^{2k-2}}$, we have shown that
\begin{equation}\label{randomresult}
C'\cdot 1(M)^{2k-1}n^{2k}\leq\binom{n}{k}^2\pi(M).
\end{equation}
Now, $1(M)=\frac{a}{n^2}$ by definition. Furthermore, the right hand side of (\ref{randomresult}) is simply the number of copies of $A_{\pi}$ in $M$ (by the definition of $\pi(M)$. Thus we have shown that $M$ contains at least $C'\frac{a^{2k-1}}{n^{2k-2}}$ copies of $A_{\pi}$, so taking this value of $C'$ and $C=3c_{\pi}$ (as above), we have proven our upper bound.

To show that this bound is sharp, take $a$ and $n$ with $Cn\leq a\leq n^2$. We may modify $a$ and $n$ by at most a constant factor so that $n|a$ and $a|n^2$. Now, suppose $\pi(1)>\pi(k)$ without loss of generality. Divide our $n\times n$ matrix $M$ into blocks of side length $\frac{a}{n}$ (so there are $\frac{n^2}{a}$ blocks on each side). Consider the $\frac{n^2}{a}$ blocks along the main (upper left to lower right) diagonal. Fill each of these blocks with ones, and fill the rest of $M$ with zeroes. The idea is to show that any copy of $A_{\pi}$ must fall entirely into one of the blocks.

How many copies of $A_{\pi}$ are contained in $M$? Recall that a copy of $A_{\pi}$ is given by a set of $k$ entries of $M$ which equal $1$, say at indices $(i_1,j_1),\ldots,(i_k,j_k)$, with $i_1<\cdots<i_k$ and the relative ordering of the $j_k$ given by $\pi$.

There are $a$ ones in $M$ (as required) and so at most $a$ choices for $(i_1,j_1)$. Let $B$ be the $\frac{a}{n}\times\frac{a}{n}$ block containing $(i_1,j_1)$. Now, $i_1<i_k$ and $j_1>j_k$ (since $\pi(1)>\pi(k)$), so we are looking for a point to the lower-left of $(i_1,j_1)$. But since all blocks containing ones are on the main diagonal, $(i_k,j_k)$ must be contained in $B$ as well.

Now, for all $r$, we have that $i_1\leq i_r\leq i_k$, and since $B$ is the only block in its row containing ones, all other entries $(i_r,j_r)$ must be contained in $B$. So for each of the remaining $k-1$ entries $(i_r,j_r)$ with $r>1$, there are at most $|B|=\left(\frac{a}{n}\right)^2$ choices. So in total there are at most
\[a\left(\frac{a}{n}\right)^{2(k-1)}=\frac{a^{2k-1}}{n^{2k-2}}\]
copies of $\pi$ in $M$. Since we only had to adjust $a,n$ by a constant factor in the start, this proves the desired sharpness bounds, completing the proof of Theorem \ref{01matrices}.
\end{proof}

En route to the proof of Theorem \ref{permsupersaturation} in the next section, we bound the number of $0\mhyphen 1$ matrices containing few copies of $A_{\pi}$.

\begin{proposition}\label{matrixcountprop}
Let $k\in\mathbb{Z}^+$ with $k>1$, $\pi\in S_k$ be fixed. There is a constant $C=C(\pi)$ such for all $m,n\geq 0$, the number of $0\mhyphen 1$ matrices of size $n\times n$ containing at most $m$ copies of $A_{\pi}$ is at most
\[\exp\left(C\left(n+\sqrt[2k-1]{mn^{2k-2}}\right)\right).\]
\end{proposition}
\begin{remark}
Notice that Proposition \ref{matrixcountprop} immediately implies the same bound on the number of permutations of length $n$ avoiding $\pi$ (since each such permutation matrix yields a $0\mhyphen 1$ matrix of size $n\times n$ avoiding $A_{\pi}$). Unfortunately, this bound is not as strong as the one we need to prove Theorem \ref{permsupersaturation}. However, we will be able to bootstrap Proposition \ref{matrixcountprop} to prove the full result by a technique involving `contracting' each permutation matrix containing few copies of $\pi$ to a smaller $0\mhyphen 1$ matrix, and then applying the Proposition to this smaller matrix.

The explanation for the expression $n+\sqrt[2k-1]{mn^{2k-2}}$ is that Theorem \ref{01matrices} implies that any $0\mhyphen 1$ matrix $M$ of size $n\times n$ containing at most $m$ copies of $A_{\pi}$ has at most $O\left(n+\sqrt[2k-1]{mn^{2k-2}}\right)$ ones. This is because if this does not hold, then letting $a$ be the number of ones in $M$, we have that both $a\gg n$ and $\frac{a^{2k-1}}{n^{2k-2}}\gg m$, which contradicts the Theorem.
\end{remark}
\begin{proof}
The proof here parallels Klazar's proof that F\"uredi-Hajnal implies Stanley-Wilf given in \cite{K}. The idea is to `contract' any $n\times n$ matrix $M$ containing few copies of $A_{\pi}$ to an $n/2\times n/2$ matrix $M'$ by dividing $M$ up into $2\times 2$ boxes and assigning each box a $1$ if and only if there are any ones in the box in $M$. If $M$ contains few copies of $A_{\pi}$, $M'$ must also, so by Theorem \ref{01matrices} it must contain few ones. But the number of possible values of $M$ given $M'$ is exponential in the number of ones of $M'$, so we may upper bound the number of possible choices for $M$ by the number of possible choices for $M'$. This will give a recursion yielding the desired bound.

In particular, let $S(n,m)$ be the set of $0\mhyphen 1$ matrices of size $n\times n$ containing at most $m$ copies of $\pi$, and let $f(n,m)=|S(n,m)|$. For a $0\mhyphen 1$ matrix $M$ of size $n\times n$ with $2|n$,  let the $2$-contraction of $M$ be the $n/2\times n/2$ $0\mhyphen 1$ matrix $M'$ such that $M'_{i,j}=0$ if and only if $M_{2i-1,2j-1}=M_{2i-1,2j}=M_{2i,2j-1}=M_{2i,2j}=0$.

Now, for each copy of $A_{\pi}$ in $M'$, there is at least one corresponding copy of $A_{\pi}$ in $M$. This is because a copy of $A_{\pi}$ in $M'$ corresponds to a choice of $k$ $1$-entries of $M'$ with relative row- and column- ordering given by $\pi$, and each $1$-entry of $M'$ corresponds (in an order-preserving way) to at least one $1$-entry of $M$. Thus $M$ contains at least as many copies of $A_{\pi}$ as its $2$-contraction $M'$, so $M'$ must also contain at most $m$ copies of $A_{\pi}$.

Therefore, if $M\in S(n,m)$, then we must have $M'\in S(n/2,m)$, where $M'$ is the $2$-contraction of $M$. Thus
\begin{equation}\label{01recursiontechnique}
f(n,m)=|S(n,m)|\leq\displaystyle\sum_{M'\in S(n/2,m)}\left|\{M:M'\text{ is the }2\text{-contraction of }M\}\right|.
\end{equation}
Now, given a matrix $M'$, how many matrices $M$ $2$-contract to $M'$? For every $0$-entry of $M'$, the corresponding four entries of $M$ must be $0$, so there are no choices to be made. For every $1$-entry of $M'$, the corresponding four entries of $M$ may be either $1$ or $0$ (but not all $0$), so there are $15$ choices for those entries of $M$. Thus there are $15^{(\#\text{ of ones in }M')}$ matrices that $2$-contract to $M'$. Combining this with (\ref{01recursiontechnique}), we obtain that
\begin{equation}\label{01recursiontechnique2}
f(n,m)\leq\displaystyle\sum_{M'\in S(n/2,m)}15^{(\#\text{ of ones in }M')}\leq f(n/2,m)\cdot 15^{\displaystyle\max_{M'\in S(n/2,m)}(\#\text{ of ones in }M')}
\end{equation}
We now apply Theorem \ref{01matrices}. For $M'\in S(n/2,m)$, we know that $M'$ has at most $m$ copies of $A_{\pi}$ by definition, so by Theorem \ref{01matrices} it must have at most
\[O\left(n/2+\sqrt[2k-1]{m(n/2)^{2k-2}}\right)=O\left(n+\sqrt[2k-1]{mn^{2k-2}}\right)\]
ones (by the discussion at the beginning of the proof). Substituting into (\ref{01recursiontechnique2}),
\begin{equation*}
f(n,m)\leq f(n/2,m)\cdot \exp(C_0(n+\sqrt[2k-1]{mn^{2k-2}})).
\end{equation*}
for some $C_0=C_0(\pi)$. This recursion is fairly easy to solve; we see that for $a\in\mathbb{Z}^{\geq 0}$
\begin{align*}
\log(f(2^a,m)) & \leq\log(f(1,m))+C_0\displaystyle\sum_{i=1}^a\left(2^i+\sqrt[2k-1]{m2^{i(2k-2)}}\right) \\ & \leq 1+C_0\left(2^{a+1}+\sqrt[2k-1]{m}\cdot \frac{2^{\frac{(a+1)(2k-2)}{2k-1}}}{2^{\frac{2k-2}{2k-1}}-1}\right) \\ & \leq (C_0+1)\left(2^{a+1}+2\cdot\sqrt[2k-1]{m\cdot 2^{(a+1)(2k-2)}}\right),
\end{align*}
where we simply summed the geometric series and used that $\log(f(1,m))\leq\log(2)\leq 1$ and that $2^{\frac{2k-2}{2k-1}}\geq\frac{3}{2}$ for $k\geq 2$. Now, $f(n,m)$ is nondecreasing in $n$ (as we may `pad' any $n\times n$ matrix with zeroes to form an $n'\times n'$ matrix with the same number of copies of $A_{\pi}$, and this process is injective). For any $n$, take $a\in\mathbb{Z}^{\geq 0}$ such that $2^{a-1}<n\leq 2^a$. Then by the previous computation,
\begin{align*}
\log(f(n,m)) & \leq\log(f(2^a,m)) \\ & \leq (C_0+1)\left(2^{a+1}+2\cdot\sqrt[2k-1]{m\cdot 2^{(a+1)(2k-2)}}\right) \\ & \leq (C_0+1)\left(4n+2\cdot\sqrt[2k-1]{m\cdot (4n)^{2k-2}}\right) \\ & \leq 8(C_0+1)\left(n+\sqrt[2k-1]{mn^{2k-2}}\right).
\end{align*}
Letting $C=8(C_0+1)$ completes the proof of Proposition \ref{matrixcountprop}.
\end{proof}

Now that we have bounded the total number of $0\mhyphen 1$ matrices that contain few copies of $A_{\pi}$, in the next section we may bound the number of permutations that contain few copies of $\pi$.
\section{Permutations with Few Copies of $\pi$}\label{randomproof}
This section will be devoted to the proof of Theorem \ref{permsupersaturation}.
\subsection{Proof of the Upper Bound}
Our proof of the upper bound of Theorem \ref{permsupersaturation} will proceed in the following steps.
\begin{enumerate}
\item For a suitable $b\in\mathbb{Z}^+$ and for any matrix permutation $\sigma$ containing few copies of $\pi$, take the $b$-contraction (analogous to the $2$-contraction in the last section) of the matrix $A_{\sigma}$ to get some $n/b\times n/b$ matrix $B_{\sigma}$.

\item $B_{\sigma}$ is a $0\mhyphen 1$ matrix containing few copies of $A_{\pi}$, so we can apply Proposition \ref{matrixcountprop}. Thus as $\sigma$ ranges over all permutations in $S_n$ containing few copies of $\pi$, $B_{\sigma}$ ranges through only a small number of distinct matrices. Thus to bound the number of $\sigma$ containing few copies of $\pi$, it suffices to bound for any matrix $B$ the number of $\sigma\in S_n$ such that $B_{\sigma}=B$.

\item We bound the desired quantity $|\{\sigma\in S_n: B_{\sigma}=B\}|$ for any $n/b\times n/b$ matrix $B$ using a simple counting argument.
\end{enumerate}
The key is that since we are only applying Proposition \ref{matrixcountprop} to an $n/b\times n/b$ matrix instead of an $n\times n$ matrix, we can obtain a much better bound (as long as $b$ is chosen accordingly).
\begin{step}
Let 
\[S_n(m,\pi):=\{\sigma\in S_n:\sigma\text{ contains at most }m\text{ copies of }\pi\}.\]
To prove the upper bound of Theorem \ref{permsupersaturation}, we would like to show that
\[|S_n(m,\pi)|\leq\exp(O(n))\cdot\max\left(1,\left(\frac{m}{n}\right)^{\frac{n}{k-1}}\right).\]
First suppose $m<n$, such that the max is dominated by the first term. Then Proposition \ref{matrixcountprop} guarantees that the number of  $0\mhyphen 1$ matrices of size $n\times n$ containing at most $m$ copies of $A_{\pi}$ is at most $\exp(O(n))$. Since each $\sigma\in S_n$ containing at most $m$ copies of $\pi$ gives rise to a permutation matrix $A_{\sigma}$ that contains at most $m$ copies of $A_{\pi}$, we see that the number of $\sigma$ containing at most $m$ copies of $\pi$ is $\exp(O(n))$, as desired.

Now suppose $m\geq n$. Take $b=\sqrt[2k-2]{\frac{m}{n}}$. Just as we took the $2$-contraction of a matrix in the proof of Proposition \ref{matrixcountprop}, we will define the $b$-contraction of any $0\mhyphen 1$ matrix of size $n\times n$. The $b$-contraction of such a matrix $A$ is the $0\mhyphen 1$ matrix $B$ such that the dimensions of $B$ are $\left\lceil\frac{n}{b}\right\rceil\times\left\lceil\frac{n}{b}\right\rceil$, and such that $B_{i,j}=1$ if and only if there exists $i',j'$ with $\left\lceil\frac{i'}{b}\right\rceil=i$ and $\left\lceil\frac{j'}{b}\right\rceil=j$ such that $A_{i',j'}=1$ (so if $A_{i',j'}=0$ for all such $i',j'$, then $B_{i,j}=0$). Let
\[n':=\left\lceil\frac{n}{b}\right\rceil=\left\lceil\sqrt[2k-2]{n^{2k-1}m^{-1}}\right\rceil\]
so that $B$ is here an $n'\times n'$ matrix.

For all $\sigma\in S_n$, let $B_{\sigma}$ be the $b$-reduction of $A_{\sigma}$.
\end{step}
\vspace{12pt}
\begin{step}
Similarly to the proof of Proposition \ref{matrixcountprop} (with the $2$-contraction), any occurrence of $A_{\pi}$ in $B_{\sigma}$ will correspond to at least one occurrence of $A_{\pi}$ in $A_{\sigma}$. This again comes from, for each $1$-entry in $B$ appearing in that occurrence of $A_{\pi}$, choosing a corresponding $1$-entry of $A$, and realizing that these $1$-entries have the same relative row- and column-ordering.

Now, we have shown that each occurrence of $A_{\pi}$ in $B_{\sigma}$ gives rise to at least one occurrence of $A_{\pi}$ in $A_{\sigma}$ (it is easy to see that these occurrences are all distinct), and the occurrences of $A_{\pi}$ in $A_{\sigma}$ correspond to occurrences of $\pi$ in $\sigma$. Thus for all $\sigma\in S_n(m,\pi)$, the matrix $B_{\sigma}$ contains at most $m$ copies of $A_{\pi}$.

By Proposition \ref{matrixcountprop} (using the fact that $m>n\geq n'$), there are at most $\exp\left(C\left(\sqrt[2k-1]{m{n'}^{2k-2}}\right)\right)$ matrices of dimension $n'\times n'$ that contain at most $m$ copies of $A_{\pi}$ (for $C=C(\pi)$). So as $\sigma$ ranges over all elements of $S_n(m,\pi)$, $B_{\sigma}$ ranges over at most
\begin{align*}
\exp\left(C\left(\sqrt[2k-1]{m{n'}^{2k-2}}\right)\right) & \leq\exp\left(C\left(\sqrt[2k-1]{m\left(2\sqrt[2k-2]{n^{2k-1}m^{-1}}\right)^{2k-2}}\right)\right) \\ & =\exp\left(C\left(\sqrt[2k-1]{2^{2k-2}mn^{2k-1}m^{-1}}\right)\right) \\ & \leq\exp(2Cn)
\end{align*}
different matrices (where we used the fact that $n'=\left\lceil\sqrt[2k-2]{n^{2k-1}m^{-1}}\right\rceil\leq 2\sqrt[2k-2]{n^{2k-1}m^{-1}}$ as $m\leq\binom{n}{k}<n^{2k-1}$). Therefore,
\begin{align}
|S_n(m,\pi)| & =\displaystyle\sum_{B\text{ size }n'\times n'}\left|\{\sigma\in S_n(m,\pi):B_{\sigma}=B\}\right| \\ & \leq\exp(2Cn)\cdot\displaystyle\max_{B\text{ size }n'\times n'}\left|\{\sigma\in S_n(m,\pi):B_{\sigma}=B\}\right| \\ & \label{contractionbound}\leq\exp(2Cn)\cdot\displaystyle\max_{B\text{ size }n'\times n'}\left|\{\sigma\in S_n:B_{\sigma}=B\}\right|.
\end{align}
\end{step}
\vspace{12pt}
\begin{step}
It only remains to bound $\displaystyle\max_B\left|\{\sigma\in S_n:B_{\sigma}=B\}\right|$ from above for all $0\mhyphen 1$ matrices $B$ of size $n'\times n'$. That is, we must prove an upper bound on the number of permutation matrices of size $n\times n$ that $b$-contract to a particular matrix. 

Now, since $A_{\sigma}$ is a permutation matrix, it has $n$ ones. By the definition of $b$-contraction, $B_{\sigma}$ must have at most $n$ ones. So in computing $\displaystyle\max_B\left|\{\sigma\in S_n:B_{\sigma}=B\}\right|$ we may assume $B$ is an $n'\times n'$ matrix with at most $n$ ones.

Let $B$ be such a matrix, and suppose there are $a_i$ ones in the $i^{th}$ row of $B$. Then $\displaystyle\sum_{i=1}^{n'}a_i\leq n$. How many choices are there for $\sigma$ such that $B_{\sigma}=B$? Consider the first row of $A_{\sigma}$, in which there is exactly one $1$. This $1$, when we take the $b$-reduction, must correspond to a $1$ of $B$ in the first row of $B$. There are $a_1$ such ones in the first row of $B$, and each one corresponds to at most $\left\lceil b\right\rceil$ entries in the first row of $A_{\sigma}$. Thus there are at most $\left\lceil b\right\rceil\cdot a_1$ ways to choose the position of the $1$ in the first row of $A_{\sigma}$--in other words, to choose $\sigma(1)$.

Similarly, the $1$-entry in the $i^{th}$ row of $A_{\sigma}$ must correspond to a $1$-entry in the $\left\lceil\frac{i}{b}\right\rceil^{th}$ row of $B$, so there are at most $\left\lceil b\right\rceil\cdot a_{\left\lceil\frac{i}{b}\right\rceil}$ ways to choose the value of $\sigma(i)$. This implies that the total number of choices for $\sigma$ such that $B_{\sigma}=B$ is at most
\begin{equation}\label{aiproductbound}
\displaystyle\prod_{i=1}^{n}\left\lceil b\right\rceil\cdot a_{\left\lceil\frac{i}{b}\right\rceil}=\left\lceil b\right\rceil^n\displaystyle\prod_{i=1}^{n}a_{\left\lceil\frac{i}{b}\right\rceil}.
\end{equation}
Now, in the sum
\begin{equation}\label{aisum}
\displaystyle\sum_{i=1}^{n}a_{\left\lceil\frac{i}{b}\right\rceil},
\end{equation}
every particular $a_j$ occurs at most $\left\lceil b\right\rceil$ times, once for every $i$ such that $bj-b<i\leq bj$. Thus (\ref{aisum}) is bounded by $\left\lceil b\right\rceil\displaystyle\sum_{j=1}^{n'}a_j\leq\left\lceil b\right\rceil\cdot n$. So by the AM-GM inequality,
\[\displaystyle\prod_{i=1}^{n}a_{\left\lceil\frac{i}{b}\right\rceil}\leq\left\lceil b\right\rceil^n.\]
Substituting into (\ref{aiproductbound}), we see that there are at most $\left\lceil b\right\rceil^{2n}$ choices for $\sigma$ such that $B_{\sigma}=B$. Finally, substituting into (\ref{contractionbound}), we have derived that
\[|S_n(m,\pi)|\leq\exp(2Cn)\left\lceil b\right\rceil^{2n}.\]
Now by definition, $b=\sqrt[2k-2]{\frac{m}{n}}$, and $m\geq n$, so $b\geq 1$ and $\left\lceil b\right\rceil\leq 2b=2\sqrt[2k-2]{\frac{m}{n}}$. Therefore,
\[\left\lceil b\right\rceil^{2n}\leq4^n\left(\frac{m}{n}\right)^{\frac{n}{k-1}}.\]
This implies that
\[|S_n(m,\pi)|\leq\exp((2C+2)n)\left(\frac{m}{n}\right)^{\frac{n}{k-1}},\]
and replacing $2C+2$ by $C$ finishes the proof of the upper bound in Theorem \ref{permsupersaturation}.
\end{step}
\subsection{Proof of the Lower Bound}
To prove the lower bound of Theorem \ref{permsupersaturation}, we must exhibit at least
\[\exp(-O(n))\cdot\max\left(1,\left(\frac{m}{n}\right)^{\frac{n}{k-1}}\right)\]
permutations $\sigma\in S_n$ such that $\sigma$ contains at most $m$ copies of $\pi$.

Suppose without loss of generality that $\pi(1)>\pi(k)$. For $m\leq n$ the all-increasing permutation avoids $\pi$, so we get a lower bound of $1$, which is sufficient.

Now suppose $m>n$. Note that $S_n(m,\pi)$ is nondecreasing in $m$ and that changing $m$ by at most a constant multiple does not change our desired lower bound by more than an exponential factor. Thus we may without loss of generality modify $m$ by a constant multiple. In particular, we may assume without loss of generality that $\frac{m}{n}$ is a $(k-1)^{st}$ power, say $a^{k-1}=\frac{m}{n}$, $a\in\mathbb{Z}^+$.

Let $S_{n,a}$ be the set of permutations $\sigma\in S_n$ such that:

$\sigma(1),\ldots,\sigma(a)$ is a permutation of $1,\ldots,a$

$\sigma(a+1),\ldots,\sigma(2a)$ is a permutation of $a+1,\ldots,2a$

$\vdots$

$\sigma\left(\left(\left\lfloor\frac{n}{a}\right\rfloor-1\right)a+1\right),\ldots,\sigma\left(\left\lfloor\frac{n}{a}\right\rfloor a\right)$ is a permutation of $\left(\left\lfloor\frac{n}{a}\right\rfloor-1\right)a+1,\ldots,\left\lfloor\frac{n}{a}\right\rfloor a$

$\sigma\left(\left\lfloor\frac{n}{a}\right\rfloor a+1\right),\ldots,n$ is a permutation of $\left\lfloor\frac{n}{a}\right\rfloor a+1,\ldots,n$.

\vspace{12pt}
It suffices to prove that
\begin{enumerate}
\item $|S_{n,a}|\geq\exp(-O(n))\left(\frac{m}{n}\right)^{\frac{n}{k-1}}$, and
\item $|S_{n,a}|\in S_n(m,\pi)$; that is, any element of $S_{n,a}$ contains at most $m$ copies of $\pi$.
\end{enumerate}

Let $n=qa+r$, $q,r\in\mathbb{Z}^{\geq 0}$, $r<a$. Then $|S_{n,a}|=\left(a!\right)^q\cdot r!$. Since $t!\geq\left(\frac{t}{e}\right)^t$ for all $t\in\mathbb{Z}^{\geq 0}$ (using $0^0=1$), we see that
\begin{align*}
|S_{n,a}| & \geq\left(\frac{a}{e}\right)^{qa}\left(\frac{r}{e}\right)^r \\ & =\left(\frac{r}{a}\right)^r\left(\frac{a}{e}\right)^n,
\end{align*}
as $qa+r=n$.
Now, the function $x^x$ is minimized for $x\in[0,1]$ when $x=\frac{1}{e}$, so $x^x\geq e^{-\frac{1}{e}}$. Thus $\left(\frac{r}{a}\right)^r=\left(\frac{r}{a}\right)^{a\frac{r}{a}}\geq \exp(-\frac{a}{e})$. Now, $m\leq\binom{n}{k}<n^k$, and therefore $a<n$. Thus
\[\left(\frac{r}{a}\right)^r\geq\exp(-n).\]

Therefore,
\[|S_{n,a}|\geq a^n\exp(-2n)=\exp(-2n)\left(\frac{m}{n}\right)^{\frac{n}{k-1}}.\]
This gives is our desired bound on $|S_{n,a}|$. It thus suffices to show that any element of $S_{n,a}$ contains at most $m$ copies of $\pi$.

Suppose $\sigma\in S_{n,a}$. At what indices can $\pi$ occur in $\sigma$? Let $\pi$ occur at some set of $k$ indices $i_1<\cdots<i_k$. Then since $\pi(1)>\pi(k)$, we must have $\sigma(i_1)>\sigma(i_k)$, while of course $i_1<i_k$. By the definition of $S_{n,a}$, this can only occur when $\left\lceil\frac{i_1}{a}\right\rceil=\left\lceil\frac{i_k}{a}\right\rceil$. Since $i_1<\cdots<i_k$, this means that there is some $t$, $0\leq t\leq q$, such that $ta+1\leq i_1<\cdots<i_k\leq (t+1)a$ (where again $qa+r=n$, $r<a$).

Given a particular value of $t$, there are thus at most $\binom{a}{k}$ choices for $(i_1,\ldots,i_k)$. However, if $t=q$, we have that $qa+1\leq i_1<\cdots<i_k\leq qa+r=n$, so there are in this case only at most $\binom{r}{k}$ choices for $(i_1,\ldots,i_k)$. Thus the total number of occurrences of $\pi$ in $\sigma$ is at most
\begin{align*}
q\binom{a}{k}+\binom{r}{k} & <qa^k+r^k \\ & \leq qa^k+ra^{k-1} \\ & \leq (qa+r)a^{k-1} \\ & =na^{k-1} \\ & =m.
\end{align*}
Thus $S_{n,a}\subseteq S_n(m,\pi)$, so we have proved the lower bound and we are done.

\section{The Fixed Hypergraph Case}\label{fixedcase}

For fixed positive integer $k$, let $\Lambda$ be a $k$-uniform hypergraph on $n$ vertices satisfying the preconditions of Theorem \ref{fixedtheorem}.  That is, for some $L$, $\Lambda$ contains a collection of $L$-vertex cliques where each of the $n$ vertices belongs to at least $\delta(\Lambda) \geq 1$ cliques in the collection and at most $\Delta(\Lambda) = O(1)$.

We would like to show that for every permutation $\pi\in S_k$,
\[Av_{n,\Lambda}(\pi)=O\left(\p{\frac{n\log^{2+\epsilon}n}{L}}^n\right)\]
for all $\epsilon > 0$.

For $L=\Theta(n^c)$ with $c \in (0,1]$, this bound is a strict improvement on the $n!$ total $n$-permutations.  For $L=n$, we match the Stanley-Wilf conjecture up to a log-exponential factor $\log(n)^{O(n)}$, which is asymptotically dominated by the linear exponential term $n^{O(n)}$ of the conjecture.\\

It may seem unnatural at first to restrict our arguments only to hypergraphs containing polynomially large cliques.  However, we see that there are very dense hypergraphs $\Lambda^*$ with $O(1)$ maximal clique size for which the number of $n$-permutations $\pi$ that are $\Lambda^*$-avoiding is $O(n)^n$.  For example, consider partitioning the vertices of $\Lambda^*$ into two parts, $\{1,\cdots,n/2\}$ and $\{n/2+1,\cdots,n\}$, and adding an edge to $\Lambda^*$ for every collection of $k$ vertices not entirely lying in a single part.  This graph will be very dense, containing ${n \choose k}-2{n/2 \choose k} \approx (1-\frac{1}{2^k}){n \choose k}$ edges.  However, there is a large class of $n$-permutations avoiding $\pi$ on these edges.  Say, without loss of generality, that $\pi(1)<\pi(k)$.  We see that all $n$-permutations $\sigma$ in which the $n/2$ largest elements belong in the first $n/2$ indices and the $n/2$ smallest elements belong in the last $n/2$ indices necessarily $\Lambda^*$-avoid $\pi$.  Each edge of $\Lambda^*$ corresponds to a sub permutation $(\sigma(x_1),\cdots,\sigma(x_k))$ in which $\sigma(x_1)>\sigma(x_k)$ and so it cannot be a copy of $\pi$.  There are $(n/2)!^2\approx \p{\frac{n}{2e}}^n = O(n)^n$ such permutations, and so there is no meaningful bound we can prove on the number of $\Lambda^*$-avoidant $n$-permutations.\\

Importantly, multipartite graphs are characterized by their small maximal cliques.  The bipartite graph we considered has maximal clique $2(k-1)$, taking $k-1$ vertices from each part.  Thus, our bounds on $\Lambda$-avoidance being contingent on $\Lambda$ containing large cliques is necessary.

\section{Hypergraph Formulation of Pattern-Avoidance}\label{formulation}

We consider a $k$-uniform hypergraph $H$ on an $n \times n$ grid of vertices $V(H)$, which we index $v(i,j)$.  Define a canonical set to be a subset of $V(H)$ of size $n$ containing exactly one vertex from each row and each column. We see that a canonical set corresponds bijectively to an $n$-permutation $\sigma$.  For a $k$-permutation $\pi$, we add edges to $H$ in such a way that each canonical set is independent if and only if its corresponding $n$-permutation is $\Lambda$-avoidant of $\pi$.  Essentially, we add an edge for each copy of $\pi$ in the vertices on columns in $E(\Lambda)$.  For all $1 \leq x_1 < x_2 < \cdots < x_k \leq n$ with $\{x_1,\cdots,x_k\} \in E(\Lambda)$ and all $1 \leq y_1< \cdots < y_k \leq n$, we have $\{v(x_1,y_{\pi(1)}),v(x_2,y_{\pi(2)}),\cdots,v(x_k,y_{\pi(k)})\} \in E(H)$.  We see that a canonical set containing the vertices of this edge would correspond to a permutation $\sigma$ that contains a copy of $\pi$ at indices $x_1,\cdots,x_k$, as desired.\\

%Another way to look at this, for each $L$-clique in $K(\Lambda)$ (the set of $L$-cliques), we define the $n \times L$ grid of vertices in the columns corresponding to the vertices of the clique to be the ``block" of the clique.  Essentially, we add an edge for each copy of the permutation matrix $P_\pi$ within a block.\\

We want to show that the number of $n$-permutations that $\Lambda$-avoid $\pi$ is $O\p{\frac{n\log^{2+\epsilon}n}{L}}^n$.  Since each permutation corresponds to a single canonical set, we want to show that the number of independent canonical sets is $O\p{\frac{n\log^{2+\epsilon}n}{L}}^n$.  In fact, our goal will be to prove a stronger claim, that the number of independent sets of size $n$, of which the independent canonical sets are a subset, is $O\p{\frac{n\log^{2+\epsilon}n}{L}}^n$.

\section{The Hypergraph Containers Lemma}\label{HC}

We introduce a version of the hypergraph container lemma due to Balogh, Morris, and Samotij \cite{BMS}.  Essentially, the container lemma is a means of placing the vertices of a hypergraph into a collection of containers $\mathcal{C}$ in such a way that each independent set in the hypergraph belongs to one of the containers.  Additionally, we ensure that no individual container contains too many vertices and that the number of containers isn't too large.  We let $\Delta_{\ell}(\cH)$ be the maximum number of hyperedges of $\cH$ that contain a given set of $\ell$ vertices.\\

\begin{proposition}[\cite{BMS} Theorem 2.2]\label{containerlemma}
Let $\mathcal{H}$ be a $k$-uniform hypergraph and let $K$ be a constant. There exists a constant $g \in (0,1)$ depending only on $k$ and $K$ such that the following holds. Suppose that for some $p\in(0,1)$ and all $\ell \in \{1, \dotsc, k\}$,
$$\Delta_{\ell}(\mathcal{H}) \le K\cdot p^{\ell-1}\cdot\frac{e(\mathcal{H})}{v(\mathcal{H})}$$
  Then, there exists a family $\cC\subseteq \mathcal P(V(\cH))$ of \emph{containers} with the following properties:
  \begin{enumerate}
  \item
    $|\cC| \leq \binom{v(\cH)}{\leq k p v(\cH)} \leq \left(\frac{e}{kp}\right)^{k p v(\cH)}$,
  \item
    $|G| \leq (1-g) \cdot v(\cH)$ for each $G \in \cC$,
  \item
    each independent set of $\cH$ is contained in some $G \in \cC$.
  \end{enumerate}
\end{proposition}

This lemma is extremely useful in bounding the number of independent sets of a hypergraph, as the number of independent sets is upper bounded by the sum of the number of independent sets in each container.  Or, in our context, the number of independent sets of size $n$ in $H$ is upper bounded by the total number of independent sets of size $n$ over all the containers.  However, a single application of the container lemma to our problem will not be strong enough for our purposes, as a single container can still contain $(1-g)|V(H)|=(1-g)n^2$ vertices and potentially have ${(1-g)n^2 \choose n} = O(n)^n$ many independent sets of size $n$.  So, we will apply the lemma recursively.  Each time we encounter a container with too many vertices, we apply the lemma to the subgraph induced by the vertices of the container and further break it up into more containers.  We do this until all the containers are sufficiently small.  Namely, we will attempt to apply the container lemma recursively until all the containers have $\leq U = \frac{Cn^2 \log^{2+\epsilon} n}{L}$ vertices, for a constant $C$ that will only be in terms of $k$ and $\pi$.  Once the containers are this small, a naive upper bound will give us that the number of size-$n$ independent sets in a container is at most ${U \choose n} = O\p{\frac{n\log^{2+\epsilon}n}{L}}^n$.\\

Unfortunately, since we know nothing about the structure of the containers, we have no guarantee that the necessary $\Delta_\ell$ bounds will hold, which are required to apply the lemma to a container.  To overcome this problem we employ a strategy similar to that used by Morris and Saxton \cite{MS}.  Consider a subgraph $G$ of $H$ induced by some subset/container of the vertices.  If we remove some of the edges of $G$ to produce a new subgraph $G'$, then every independent set of $G$ will also be an independent set of $G'$.  So, if we apply the containers lemma to $G'$, the resulting containers will also cover all the independent sets of $G$.  This will be our approach: for each container-induced subgraph $G$ with more than $U$ vertices, we construct a subgraph $G' \subseteq G$ on the same vertex set with some edges of $G$ removed.  We construct $G'$ to satisfy the preconditions of Proposition \ref{containerlemma} for a sufficiently small $p$, and so, we can break up $G'$ into containers using the proposition and recurse.  In this recursive process, we guarantee that all of the independent sets in the original $H$ are preserved.  We also ensure that we will not have too many containers in the end because we keep $p$ small.\\

Formally, in Section \ref{RL}, we prove the following lemma

\begin{lemma}\label{recursivelemma}
Let $\gamma = \frac{1}{1-g}$, where $g$ is defined in Proposition \ref{containerlemma}.  For the hypergraph $H$ defined in Section \ref{formulation}, consider a subgraph $G \subseteq H$ induced by some subset of the vertices, where
$$Cn\gamma^{t-1} < |V(G)| \leq Cn\gamma^{t}$$
for some constant $C$ and some $t \geq t_0+1$ with $Cn\gamma^{t_0}=U$.  There exists a subgraph $G' \subseteq G$ on the same vertex set such that
$$\Delta_{\ell}(G') \le K\cdot p_t^{\ell-1}\cdot\frac{|E(G')|}{|V(G)|}$$
for all $\ell \in \{1,\cdots,k\}$, where $p_t = \frac{n}{t^{2+\epsilon}|V(G)|}$\\
\end{lemma}

This lemma enables the proof of Theorem \ref{fixedtheorem} using the recursive hypergraph containers strategy.

\begin{thm}{\bf \ref{fixedtheorem}}
For every permutation $\pi$ and any $\epsilon > 0$, the number of $n$-permutations $\Lambda$-avoiding $\pi$ is $O\p{\frac{n\log^{2+\epsilon}n}{L}}^n$ as long as $\Lambda$ is $k$-uniform and satisfies the following:\\

$\Lambda$ contains a collection of $L$-vertex cliques where each of the $n$ vertices belongs to at least $\delta(\Lambda) \geq 1$ cliques in the collection and at most $\Delta(\Lambda) = O(1)$.
\end{thm}

\begin{proof}[Proof of Theorem \ref{fixedtheorem}]
As stated earlier in this section, to prove Theorem \ref{fixedtheorem} it is sufficient to prove that the hypergraph $H$ has at most $O\p{\frac{n\log^{2+\epsilon}n}{L}}^n$ independent sets of size $n$.\\

Lemma \ref{recursivelemma} shows that, for a general container $G$, we can apply the container lemma for a certain $p=p_t$ depending on the size of $G$, and further split $G$ into more containers.  Starting from the original graph $H$, we can repeat this process recursively until all of our containers have $\leq U = \frac{Cn^2 \log^{2+\epsilon} n}{L}$ vertices.  We are trying to count the number of independent sets of size $n$ in the original hypergraph and we know every independent set in the original graph is a subset of one of these containers.  Each container of size $\leq U$ has $\leq {U \choose n} \leq \p{\frac{eU}{n}}^n = O\p{\frac{n\log^{2+\epsilon}n}{L}}^n$ subsets of size $n$, and so the number of independent sets of size $n$ in this container is also bounded by this amount.  Therefore, all that remains to show is that the number of containers is singly exponential in $n$.  That is, upper bounded by $c^n$ for some $c=O(1)$.  If we can show this, then we will have

\begin{align*}
&(\text{number of size-$n$ independent sets in $H$})\\
&\leq \sum_{\text{containers }C}(\text{number of size-$n$ independent sets in $C$})\\
 &= (\text{number of conatiners})\cdot O\p{\frac{n\log^{2+\epsilon}n}{L}}^n\\
 & = c^n \cdot O\p{\frac{n\log^{2+\epsilon}n}{L}}^n  = O\p{\frac{n\log^{2+\epsilon}n}{L}}^n\\
\end{align*}

Say somewhere along this branching recursive process, we encounter a container $G$ that we want to split into further containers with $Cn\gamma^{t-1} < v(G) \leq Cn\gamma^{t}$ and $t \geq t_0+1$.  From Lemma \ref{recursivelemma}, we know we can apply the container lemma with $p=p_t = \frac{n}{t^{2+\epsilon}|V(G)|}$
and split $G$ into at most 
$$\p{\frac{e}{kp_t}}^{kp_tv(G)} = \p{\frac{et^{2+\epsilon}|V(G)|}{kn}}^{kn/t^{2+\epsilon}}  \leq \p{\frac{et^{2+\epsilon}Cn\gamma^t}{kn}}^{kn/t^{2+\epsilon}} = \p{\frac{et^{2+\epsilon}C\gamma^t}{k}}^{kn/t^{2+\epsilon}}$$
containers.  Additionally, we know that all of the resulting containers will contain at most $(1-\gamma)v(G) \leq Cn\gamma^{t-1}$ vertices.  We will subsequently break down these child containers using $p=p_s$ for some $s \leq t-1$.  Say $T=\log_\gamma(n/C)$ or equivalently $Cn\gamma^T = n^2$.  In the worst case, after we break up $H$ with $p=p_T$, we break up all of $H$'s child containers with $p=p_{T-1}$, all of $H$'s grandchild containers with $p=p_{T-2}$, etc all the way to $p=p_{t_0+1}$.  However, we can never encounter two consecutive generations of containers on which we apply the containers lemma with the same $p_t$; $t$ is always strictly decreasing.  Thus, the number of containers we have at the end is at most

$$\prod_{t=t_0+1}^T \p{\frac{e}{kp_t}}^{kp_tv(G)} \leq \prod_{t=t_0+1}^T \p{\frac{et^{2+\epsilon}C\gamma^t}{k}}^{kn/t^{2+\epsilon}}$$
Defining $A = \gamma \p{\frac{et^{2+\epsilon}C}{k}}^{1/t}$,
$$= \prod_{t=t_0+1}^T (A^t)^{kn/t^{2+\epsilon}} = \p{A^{k\sum_{t=t_0+1}^T \frac{1}{t^{1+\epsilon}}}}^n$$
We note that $A^{k\sum_{t=t_0+1}^T \frac{1}{t^{1+\epsilon}}}=O(1)$ as $\sum_{t=1}^\infty \frac{1}{t^{1+\epsilon}}$ is a convergent sum and $t^{(2+\epsilon)/t}$ has a finite upper bound for all $t$.  Moreover, $A$ is a constant in terms of $C$ and $\gamma$ which are only in terms of $\pi$ and $k$, and have no dependency on $n$.  Thus, the number of containers at the end of the branching process is singly exponential in $n$, as desired.
\end{proof}

\section{Block Decomposition of $G$}\label{RL}

We now demonstrate how to construct a subgraph $G'$ of $G$ that satisfies the preconditions of the containers lemma and enables the recursive argument.

\begin{lemma*}{\bf \ref{recursivelemma}}
Let $\gamma = \frac{1}{1-g}$, where $g$ is defined in Proposition \ref{containerlemma}.  For the hypergraph $H$ defined in Section \ref{formulation}, consider a subgraph $G \subseteq H$ induced by some subset of the vertices, where
$$Cn\gamma^{t-1} < |V(G)| \leq Cn\gamma^{t}$$
for some constant $C$ and some $t \geq t_0+1$ with $Cn\gamma^{t_0}=U$.  There exists a subgraph $G' \subseteq G$ on the same vertex set such that
$$\Delta_{\ell}(G') \le K\cdot p_t^{\ell-1}\cdot\frac{|E(G')|}{|V(G)|}$$
for all $\ell \in \{1,\cdots,k\}$, where $p_t = \frac{n}{t^{2+\epsilon}|V(G)|}$\\
\end{lemma*}

\begin{proof} Note that, for any $G'$ we construct, we will have $\Delta_{k}(G')=1$ as a set of $k$ vertices belonging to multiple edges would imply that we have duplicate edges.  So, in order to satisfy the condition for $\ell = k$, we have $1 \le K\cdot p_t^{k-1}\cdot\frac{|E(G')|}{|V(G)|}$, and so we must have $|E(G')| \geq \frac{|V(G)|}{Kp_t^{k-1}}$.  We define
$$N = \frac{|V(G)|}{Kp_t^{k-1}}$$
and so, we will construct a $G'$ with $|E(G')| \geq N$.  Therefore, it will be sufficient to construct a $G'$ satisfying
\begin{equation}\label{precond}
\Delta_{\ell}(G') \le K\cdot p_t^{\ell-1}\cdot\frac{N}{|V(G)|} = \frac{1}{p_t^{k - \ell}} %= \p{\frac{t^{2+\epsilon}|V(G)|}{n}}^{k - \ell}
\end{equation}
for all $1 \le \ell \le k$.\\

To construct a subset of the edges of $G$ satisfying \eqref{precond}, we must take advantage of some structural regularity of $G$.  However, on the whole, all we know about $G$ is roughly how many vertices it contains.  Yet, we can take advantage of the fact that, specifically on the vertices of $G$ corresponding to some clique in $\Lambda$, the edges are well-behaved and will lend themselves to an intricate construction.  Namely, for each of the $L$-cliques in $\Lambda$, we define its ``block" to be the subgraph of $G$ induced by the set of vertices in $V(G)$ belonging to the $L$ columns corresponding to this clique.  We will construct $G'$ by constructing a block subgraph $B'$ for every block $B$ and then taking $G'$ to be the union of all the $B'$.\\

We call a block $B$ rich if
$$|V(B)| \geq d = \sqrt{Lt^{2 + \epsilon}|V(G)|}$$
We will show that for every rich block $B$, there exists a subgraph $B'$ of $B$, on the same vertex set with 
 $$|E(B')|=N_B=\frac{2\Delta(\Lambda)}{\delta(\Lambda)}\cdot \frac{|V(B)|}{Kp_t^{k-1}}$$
We will also show that $B'$ can be chosen so that
$$\Delta_{\ell}(B') \le \frac{1}{\Delta(\Lambda)p^{k-\ell}}$$

If we can prove that such a $B'$ exists for every rich block $B$, then we can construct $G'$ by taking the union of all the $B'$.  We see that, for any collection of $\ell$ vertices $v_1,\cdots,v_\ell$,

$$\deg_{G'}(v_1,\cdots,v_\ell) \le \sum_{\text{rich blocks }B} \deg_{B'}(v_1,\cdots,v_\ell) \leq \frac{\Delta(\Lambda)}{\Delta(\Lambda)p^{k-\ell}}$$

since any collection of $\ell$ vertices, as well as any single vertex, belongs to at most $\Delta(\Lambda)$ blocks.  And so, $\Delta_{\ell}(G') \le \frac{1}{p_t^{k - \ell}}$ as desired.  We also see that we will have at least $N$ edges in the union of the $B'$ because

\begin{align*}
\card{\bigcup_{\text{rich blocks }B} E(B')} &\geq \frac{1}{\Delta(\Lambda)} \sum_{\text{rich blocks }B} \card{E(B')}\\
&= \frac{1}{\Delta(\Lambda)} \sum_{\text{rich blocks }B} N_B\\
&= \frac{1}{\Delta(\Lambda)} \sum_{\text{rich blocks }B} \p{\frac{2\Delta(\Lambda)}{\delta(\Lambda)Kp_t^{k-1}}} |V(B)|\\
&= \frac{2}{\delta(\Lambda)Kp_t^{k-1}} \cdot \sum_{\text{rich blocks }B} |V(B)|\\
\end{align*}

and we see

$$\sum_{\text{rich blocks }B} |V(B)| = \sum_{\text{blocks }B} |V(B)| - \sum_{\text{unrich blocks }B} |V(B)|$$

where

$$\sum_{\text{blocks }B} |V(B)| \geq \delta(\Lambda) |V(G)|$$

since each vertex belongs to at least $ \delta(\Lambda)$ blocks, and

$$\sum_{\text{unrich blocks }B} |V(B)| \leq d(\text{number of unrich blocks}) \leq d(\text{number of blocks})$$

Now, since $Cn\gamma^{t-1}<|V(G)| \leq n^2$

$$d=\sqrt{Lt^{2 + \epsilon}|V(G)|} < \sqrt{L|V(G)|\log_\gamma^{2 + \epsilon} \p{\frac{\gamma |V(G)|}{Cn}}} \leq \sqrt{L|V(G)|\log_\gamma^{2 + \epsilon} n}$$

for $C \geq \gamma$.  And since each of the $n$ vertices in $\Lambda$ belongs to at most $\Delta(\Lambda)$ of the size $L$ cliques, the number of $L$-cliques, which is the number of blocks, is at most $\Delta(\Lambda)n/L$.  So,

\begin{align*}
\sum_{\text{rich blocks }B} |V(B)| &\geq \sum_{\text{blocks }B} |V(B)| -d(\text{number of blocks})\\
& \geq \delta(\Lambda) |V(G)| - \p{\sqrt{L |V(G)| \log_\gamma^{2 + \epsilon} n}}(\Delta(\Lambda)n/L)\\
& \geq \delta(\Lambda) |V(G)| /2
\end{align*}

because

\begin{align*}
|V(G)| &\geq U = \frac{Cn^2 \log^{2+\epsilon} n}{L}\\
\therefore \sqrt{|V(G)|} &\geq \sqrt{\frac{Cn^2 \log^{2+\epsilon} n}{L}}\\
\therefore |V(G)| &\geq \frac{n}{L} \cdot \sqrt{|V(G)| \cdot CL \log^{2+\epsilon} n}\\
\therefore \delta(\Lambda) |V(G)| /2 &\geq \p{\sqrt{L |V(G)| \log_\gamma^{2 + \epsilon} n}}(\Delta(\Lambda)n/L)\\
\end{align*}

for $C \geq \frac{\Delta(\Lambda)^2}{(\delta(\Lambda)/2)^2\log(\gamma)}$, which is not in terms of $n$ and is therefore a valid bound on the constant $C$.  And so,

\begin{align*}
\card{\bigcup_{\text{rich blocks }B} E(B')} &\geq \frac{2}{\delta(\Lambda)Kp_t^{k-1}}  \cdot \sum_{\text{rich blocks }B} |V(B)|\\
&\geq \frac{2}{\delta(\Lambda)Kp_t^{k-1}} \cdot \delta(\Lambda) |V(G)| /2\\
& = N
\end{align*}

as desired.\end{proof}

\section{Supersaturation on the Rich Blocks} \label{SRB}

From the previous section, we showed that, to prove Lemma \ref{recursivelemma}, it was sufficient to show the following lemma about rich blocks.  Again, we define blocks to be the subgraph induced by the vertices of $G$ belonging to a certain collection of $L$ columns of the $n \times n$ grid.  These $L$ columns represent a clique in the avoidance hypergraph $\Lambda$.  So, for any $k$ of these $L$ columns $x_1 < x_2 < \cdots < x_k$ and any $1 \leq y_1< \cdots < y_k \leq n$, we have $\{v(x_1,y_{\pi(1)}),v(x_2,y_{\pi(2)}),\cdots,v(x_k,y_{\pi(k)})\} \in E(H)$.  We will have an edge in our block $B$ for every collection of $k$ such vertices that belong to $G$.  Now, we state the lemma.

\begin{lemma}
For a block subgraph $B \subseteq G \subseteq H$ with
$$|V(B)| \geq d =\sqrt{Lt^{2 + \epsilon}|V(G)|}$$
and
$$Cn\gamma^{t-1} < |V(G)| \leq Cn\gamma^{t}$$
for some $t \geq t_0+1$, there exists a subgraph $B' \subseteq B$ on the same vertex set such that
 \begin{equation}\label{NBRef}
 |E(B')|=N_B=\frac{2\Delta(\Lambda)}{\delta(\Lambda)}\cdot \frac{|V(B)|}{Kp_t^{k-1}}
 \end{equation}
and
\begin{equation}\label{delBound}
\Delta_{\ell}(B') \le \frac{1}{\Delta(\Lambda)p^{k-\ell}}
\end{equation}
for all $\ell \in \{1,\cdots,k\}$, where $p_t = \frac{n}{t^{2+\epsilon}|V(G)|}$ and $\gamma = \frac{1}{1-g}$, where $g$ is defined in Proposition \ref{containerlemma}.\\
\end{lemma}

\begin{proof} From our definition of $N_B$ in \eqref{NBRef}, we can rewrite \eqref{delBound} as
$$\Delta_{\ell}(B') \le \frac{\delta(\Lambda)}{2\Delta(\Lambda)^2}\cdot \frac{Kp_t^{\ell-1} N_B}{|V(B)|}$$

We start our construction of $B'$ with the hypergraph $B_0$ on the vertices of $B$ with no edges.  We then iteratively construct $B_1,B_2,\cdots,B_{N_B}$ where we construct $B_{i+1}$ by adding an edge to $B_i$.  $B_{N_B}$ will be our $B'$.\\

For every $\ell \in [1,k-1]$ and every $i \in [0,N_B-1]$, we define the dangerous set $D_\ell(B_i)$ to be the set of all sets of $\ell$ vertices $\{v_1,\cdots,v_\ell\}$ where $$|\{E \in E(B_i) | \{v_1,\cdots,v_\ell\} \subseteq E\}| \geq \frac{\delta(\Lambda)}{2\Delta(\Lambda)^2}\cdot \frac{Kp_t^{\ell-1} N_B}{|V(B)|}-1$$

Now, we say that an edge $E \in E(B)$ is $i$-safe if $F \not \in D_{|F|}(B_i)$ for every nonempty, strict subset $F \subset E$. Our goal for all $i$ will be to construct $B_{i+1}$ by adding an $i$-safe edge to $B_i$ that is not already in $E(B_i)$.  If this is always possible, we see that, for all $\ell \in \{1,\cdots,k-1\}$,

\begin{align*}
\Delta_\ell(B_{i+1}) \leq \max\p{\Delta_\ell(B_i),  \frac{\delta(\Lambda)}{2\Delta(\Lambda)^2}\cdot \frac{Kp_t^{\ell-1} N_B}{|V(B)|}-1+1}
\end{align*}
and therefore, we can show inductively that the $B_{N_B}$ we construct will satisfy $\Delta_\ell(B_{N_B}) \leq   \frac{\delta(\Lambda)}{2\Delta(\Lambda)^2}\cdot \frac{Kp_t^{\ell-1} N_B}{|V(B)|}$ and be a valid choice for $B'$, as desired.\\

In order to show there is always an $i$-safe edge $E$ not already in $E(B_i)$, it is sufficient to show that the number of $i$-safe edges is $\geq N_B$, meaning that, by the pigeonhole principle, one of them is not already in $E(B_i)$.  Let $Z$ be the number of $i$-safe edges in $B$.  We want to show $Z \geq N_B$.  The vertices of $B$ belong to an $n \times L$ matrix grid.  We define $S$ to be the set of vertices in $B$ that belong to a random submatrix, selecting each column independently with probability $q$ and each row independently with probability $\frac{Lq}{n}$, for a fixed $q \in (0,1]$.  In expectation, we select $qL$ rows and $qL$ columns.  Also note that the probability that a single vertex is included in $S$ is $\frac{Lq^2}{n}$ as both its row and column need to be selected.\\

Then, we generate another vertex subset $S' \subseteq S$.  We start with $S' = S$ and iteratively scan $S'$ for subsets of vertices $F \in D_{|F|}(B_i)$.  If we find such an $F$, we remove one of the vertices in $F$ from $S'$ and restart the scanning process.  We terminate once no subset of $S'$ belongs to $D_{|F|}(B_i)$.  We note that the total number of vertices of $S$ that we deleted is at most the number of $F \in D_{|F|}(B_i)$ contained in $S$ originally, as some vertex deletions may have destroyed multiple dangerous $F$.\\

Now, we consider the subgraph $R$ induced by $S'$ and define the random variable $X$ to be the number of $i$-safe edges in $R$.  Since we removed a vertex from every dangerous $F$ in $S'$, there will be no dangerous $F$ in $V(R)$ and every edge in $R$ is $i$-safe.  So, we have $X = |E(R)|$.\\

The probability that any $i$-safe edge in $B$ belongs to $R$ is $\leq \p{\frac{Lq^2}{n}}^{k}$, as each of the $k$ vertices in the edge belongs to $S$ with probability $\frac{Lq^2}{n}$ independently, as they all occupy separate rows and columns.  So, by linearity of expectation, we can upper bound

$$\mathbb{E}[X] \leq Z\p{\frac{Lq^2}{n}}^{k}$$

Now, to show that $Z \geq N_B$, it remains to give a sufficiently strong lower bound for $\mathbb{E}[X]$.  The main tool for during so is the F\"uredi-Hajnal conjecture (proved in 2004) which also formed the backbone of the original proof of the Stanley-Wilf conjecture.  The F\"uredi-Hajnal conjecture \cite{FH} states that any $0\mhyphen 1$ matrix $A$ of size $x \times x$ that avoids a permutation matrix $P$ can have at most $c_Px$ 1-entries, for a constant $c_P$ only in terms of $P$.  The direct implication of this theorem in the hypergraph setting is this.  For a hypergraph with an $x \times x$ grid of vertices, and edges corresponding to the copies of $P$ on this grid, any independent set of this graph has at most $c_Px$ vertices.  Using this, we can lower bound the number of edges in $R$ using a supersaturation argument.\\

Let $x = \max(\text{number of rows selected in }S, \text{number of columns selected in }S)$.  So, all of the vertices in $R$ belong to an $x \times x$ subgrid.  We claim that, by F\"uredi-Hajnal, $|E(R)| \geq |V(R)|-c_Px$.  While $R$ has more than $c_Px$ vertices, we can find an edge in $R$ and delete one of the vertices in that edge.  This decreases $|V(R)|$ by 1, and decreases $|E(R)|$ by at least 1.  Repeating this process until the number of vertices left in $R$ is $c_Px$, we must have removed at least $|V(R)|-c_Px$ edges which were originally in $R$.  Thus, by linearity of expectation,

$$\mathbb{E}[|E(R)|] \geq \mathbb{E}[|V(R)|-c_Px] = \mathbb{E}[|V(R)|]-c_P\mathbb{E}[x]$$
Now, 
\begin{align*}
\mathbb{E}[|V(R)|] &= \mathbb{E}[|S'|] = \mathbb{E}[|S| - \text{at most 1 for each dangerous set in }S]\\
&\geq \mathbb{E}[|S|] - \sum_{\ell = 1}^{k-1}\sum_{F \in D_\ell(B_i)} \text{Pr}[F \subseteq S]\\
&= \frac{Lq^2}{n} |V(B)| - \sum_{\ell = 1}^{k-1} |D_\ell(B_i)| \cdot \p{\frac{Lq^2}{n}}^{\ell}
\end{align*}
 and 
 \begin{align*}
 \mathbb{E}[x] &= \mathbb{E}[\max(\text{number of rows selected}, \text{number of columns selected})]\\
 &\leq \mathbb{E}[\text{number of rows selected}]+\mathbb{E}[ \text{number of columns selected}]\\
 &= \frac{Lq}{n}\cdot n + q\cdot L = 2qL
 \end{align*}
Therefore,
$$Z\p{\frac{Lq^2}{n}}^{k} \geq \mathbb{E}[|E(R)|] \geq \frac{Lq^2}{n} |V(B)| - \sum_{\ell = 1}^{k-1} |D_\ell(B_i)| \cdot \p{\frac{Lq^2}{n}}^{\ell} - 2qc_PL$$

We have that $|V(B)| > U = Cn\gamma^{t_0}$ for some constant $C$.  So, we take $C > 4c_P$, which is only in terms of $\pi$ and is therefore a valid constraint on $C$.  Setting $q = \frac{4c_Pn}{|V(B)|} < 1$, we have

$$\frac{Lq^2}{n} |V(B)| - 2qc_PL \geq \frac{Lq^2}{2n} |V(B)|$$
and

$$Z\p{\frac{Lq^2}{n}}^{k} \geq \mathbb{E}[|E(R)|] \geq \frac{Lq^2}{2n} |V(B)| - \sum_{\ell = 1}^{k-1} |D_\ell(B_i)| \cdot \p{\frac{Lq^2}{n}}^{\ell}$$

So, in order to show $Z \geq N_B$, it is sufficient to show

\begin{equation}\label{suffshow}
\frac{Lq^2}{2n} |V(B)| - \sum_{\ell = 1}^{k-1} |D_\ell(B_i)| \cdot \p{\frac{Lq^2}{n}}^{\ell} \geq N_B\p{\frac{Lq^2}{n}}^{k}=  \frac{2\Delta(\Lambda)}{\delta(\Lambda)}\cdot \frac{|V(B)|}{Kp_t^{k-1}}\p{\frac{Lq^2}{n}}^{k}
\end{equation}

We can bound $|D_\ell(B_i)|$ by double counting $F,E$ pairs where $$F = \{v_1,\cdots,v_\ell\} \subseteq E \in E(B_i)$$

For an upper bound, we know there are $i \leq N_B$ ways to choose $E \in E(B_i)$ and there are ${k \choose \ell} \leq 2^k$ ways to choose an $F$ belonging to that $E$.  For a lower bound, each $F \in D_{\ell}(B_i)$ belongs to at least $\frac{\delta(\Lambda)}{2\Delta(\Lambda)^2}\cdot \frac{Kp_t^{\ell-1} N_B}{|V(B)|} - 1$ many edges and each $F \not \in D_{\ell}(B_i)$ belongs to at least 0 edges.  So,

$$2^kN_B \geq \text{number of }F,E\text{ pairs} \geq |D_\ell(B_i)|\p{ \frac{\delta(\Lambda)}{2\Delta(\Lambda)^2}\cdot \frac{Kp_t^{\ell-1} N_B}{|V(B)|} - 1}$$
$$\therefore |D_\ell(B_i)| \leq \frac{2^{k+2}\Delta(\Lambda)^2|V(B)|}{\delta(\Lambda)Kp_t^{\ell-1}}$$
as $\frac{\delta(\Lambda)}{2\Delta(\Lambda)^2}\cdot \frac{Kp_t^{\ell-1} N_B}{|V(B)|} - 1 \geq \frac{\delta(\Lambda)}{4\Delta(\Lambda)^2}\cdot \frac{Kp_t^{\ell-1} N_B}{|V(B)|}$ for sufficiently large $K$.  Inserting this bound into \eqref{suffshow}, it is sufficient to show

$$\frac{Lq^2}{2n} |V(B)| - \sum_{\ell = 1}^{k-1} \frac{2^{k+2}\Delta(\Lambda)^2|V(B)|}{\delta(\Lambda)Kp_t^{\ell-1}} \cdot \p{\frac{Lq^2}{n}}^{\ell} \geq \frac{2\Delta(\Lambda)}{\delta(\Lambda)}\cdot \frac{|V(B)|}{Kp_t^{k-1}}\p{\frac{Lq^2}{n}}^{k}$$
dividing through by $\frac{Lq^2|V(B)|}{n}$
$$\frac{1}{2}  - \frac{2^{k+2}\Delta(\Lambda)^2}{\delta(\Lambda)K}  \sum_{\ell = 1}^{k-1} \p{\frac{Lq^2}{np_t}}^{\ell-1} \geq \frac{2\Delta(\Lambda)}{K\delta(\Lambda)}\p{\frac{Lq^2}{np_t}}^{k-1}$$
and so we want
\begin{equation}\label{finalboy}
K \geq \frac{2^{k+3}\Delta(\Lambda)^2}{\delta(\Lambda)} \sum_{\ell = 0}^{k-1} \p{\frac{Lq^2}{np_t}}^{\ell}
\end{equation}
Lastly, since we set $q = \frac{4c_Pn}{|V(B)|}$, we have $\frac{Lq^2}{np_t} = \frac{16c_P^2Ln}{p_t|V(B)|^2}$. And since $B$ is given to be a rich block, we have
$$|V(B)|^2 \geq Lt^{2 + \epsilon}|V(G)|=Ln/\p{\frac{n}{t^{2 + \epsilon}|V(G)|}} = \frac{Ln}{p_t}$$

Thus, $\frac{Lq^2}{np_t}  \leq 16c_P^2$ and we can set $K = \frac{2^{k+3}\Delta(\Lambda)^2}{\delta(\Lambda)} \cdot \frac{(16c_P^2)^k-1}{16c_P^2-1}$, which satisfies \eqref{finalboy} and is not in terms of $n$, making it a valid definition for the constant $K$.
\end{proof}

\section{Conclusion}\label{conclusion}

We have managed to show that the number of $n$-permutations $\Lambda$-avoiding $\pi$ is $O\p{\frac{n\log^{2+\epsilon}n}{L}}^n$ only relying on the fact that $\Lambda$ contains a certain collection of size-$L$ cliques.  This bound holds for positive $\epsilon$ arbitrarily close to $0$. When $L$ is polynomial in $n$, that is $L=\Theta(n^c)$ with $c \in (0,1]$, this bound is a strict improvement on the $n!=O(n)^n$ total $n$-permutations.  For $L=n$, we are a log-exponential factor off from the Stanley-Wilf conjecture.\\

Our matching bound for when $\Lambda$ is a random hypergraph with edge probability $\alpha$ of $\exp(O(n))\alpha^{-\frac{n}{k-1}}$ is therefore more general in many ways, as there are no cliques of polynomial size in $n$ w.h.p. in such a random graph.  This is expected as the weakest part of our argument came from the deterministic nature of $\Lambda$.  When we are bounding the sum of the vertices in the rich blocks,

$$\sum_{\text{rich blocks }B} |V(B)| = \sum_{\text{blocks }B} |V(B)| - \sum_{\text{unrich blocks }B} |V(B)|$$

\noindent the best bound for the unrich blocks 
$$\sum_{\text{unrich blocks }B} |V(B)| \leq d(\text{number of unrich blocks}) \leq d(\text{number of blocks})$$

assumes that all the blocks are unrich, accounting for the worst deterministic case.  When the locations of the blocks are randomized, we can make a stronger statement in expectation.  However, such a reliance on large cliques in the fixed $\Lambda$ case is necessary to achieve any meaningful bound, as we showed there are dense multipartite hypergraphs $\Lambda^*$ which are avoided by $O(n)^n$ permutations of length $n$, but which have constant maximal clique.  This gives us hope that the conditions we place on the fixed $\Lambda$ are relatively tight.\\

An open problem is to remove the $\log^{2+\epsilon}n$ term from the bound.  The term comes from the use of hypergraph containers in a recursive branching fashion.  Each container in the tree is broken down using the containers lemma as a black box, necessitating this term.  It may be removable by reworking the arguments of the containers lemma to tailor to this recursive usage, which would improve our bound especially for $L=\Theta(n)$.

\section{Acknowledgements}
The authors would like to thank Asaf Ferber for his mentorship throughout this research.  His teachings and advice were invaluable. They would also like to thank the reviewer for their helpful advice on improving the paper.

\end{document}